\documentclass{amsart}

\usepackage{amsmath,amsfonts,amssymb}
\newtheorem{thm}{Theorem}
\newtheorem{rk}{Remark}
\newtheorem{prop}{Proposition}
\newtheorem{clly}{Corollary}
\newtheorem{lemma}{Lemma}
\newtheorem{defi}{Definition}
\newcommand{\pf}{{\flushleft{\bf Proof: }}}
\newcommand{\R}{{\mathbb{R}}}
\newcommand{\C}{{\mathbb{C}}}
\newcommand{\Z}{{\mathbb{Z}}}
\newcommand{\N}{{\mathbb{N}}}
\newcommand{\T}{{\mathbb{T}}}
\newcommand{\D}{{\mathbb{D}}}
\DeclareMathOperator{\id}{Id}
\begin{document}
\title{Dynamics of covering maps of the annulus I: Semiconjugacies.}
\author{J.Iglesias, A.Portela, A.Rovella and J.Xavier}

\address{J. Iglesias, Universidad de La Rep\'ublica. Facultad de Ingenieria. IMERL. Julio
Herrera y Reissig 565. C.P. 11300. Montevideo, Uruguay}
\email{jorgei@fing.edu.uy }

\address{A. Portela, Universidad de La Rep\'ublica. Facultad de Ingenieria. IMERL. Julio
Herrera y Reissig 565. C.P. 11300. Montevideo, Uruguay }
\email{aldo@fing.edu.uy }

\address{A. Rovella, Universidad de La Rep\'ublica. Facultad de Ciencias. CMAT.
Igu\'a 4225. C.P. 11400. Montevideo, Uruguay}
\email{leva@cmat.edu.uy}
\
\address{J. Xavier, Universidad de La Rep\'ublica. Facultad de Ingenieria. IMERL. Julio
Herrera y Reissig 565. C.P. 11300. Montevideo, Uruguay }
\email{jxavier@fing.edu.uy }
\begin{abstract}
It is often the case that a covering map of the open annulus is semiconjugate to a map of the circle of the same degree.
We investigate this possibility and its consequences on the dynamics. In particular, we address the problem of the classification up to conjugacy.
However, there are examples which are not semiconjugate to a map of the circle, and this opens new questions.
\end{abstract}
\maketitle

\section{Introduction.}
Let $A$ be the open annulus $(0,1)\times S^1$.
If $f:A\to A$ is a continuous function, then the homomorphism $f_{*}$ induced by $f$ on the first homology group $H_1(A,\Z)\equiv \Z$, is $n\to dn$, for some integer $d$. This number $d$ is called the degree of $f$.  If $f$ is a covering map, any point $x\in A$ has an open neighborhood $U$ such that $f^{-1} (U)$ is a disjoint union of $|d|$ open sets, each of which is mapped homeomorphically onto $U$ by $f$.

In this article we consider the dynamics of covering maps $f:A \to A$ of degree $d$, with $|d|>1$. Our interest is focused on the existence of a semiconjugacy with $m_d(z)=z^d$ acting on $S^1$. In general, looking for semiconjugacies with maps with known features is useful to classify, to find periodic orbits, to calculate entropy.

\begin{defi}
\label{sc}
A continuous map $f:T\to T$, is semiconjugate to $g:T'\to T'$ (where $T$ and $T'$ stand for the annulus or the circle) if there exists a continuous map $h: T\to T'$ such that $hf=gh$ and $h_{*}$ is an isomorphism.
\end{defi}

The case of coverings where $d=\pm 1$ (i.e annulus homeomorphisms) has been extensively studied. In particular, much work has been devoted to use the rotations as models for the dynamics of other homeomorphisms. That is, is it possible to decide if a homeomorphism of the annulus that is isotopic to the identity is conjugate or semi-conjugate to a rotation?

The case of the circle is classical: An orientation-preserving homeomorphism of the circle is semi - conjugate to an irrational rotation if and only if its rotation number is irrational,
and if and only if it has no periodic points \cite{poinc}. Poincar\'e's construction of the rotation number can be generalized to a {\it rotation set} for homeomorphisms of the (open or closed) annulus that are isotopic to the
identity (see, for example \cite{becrlerp}).  If one focuses on
pseudo-rotations, the class of annulus homeomorphisms whose rotation set is reduced to a single number $\alpha$, it is natural to ask the following question: how much does the dynamics of a
pseudo-rotation of irrational angle
$\alpha$ look like the dynamics of the rigid rotation of angle $\alpha$? (see \cite{becrlerp}, \cite{becrler} for results on this subject). Even in the compact case (i.e irrational pseudo-rotations of the closed annulus) it is known
that the dynamics is not conjugate to that of the rigid rotation.  Furthermore, examples that are not even semi-conjugate to a rotation in the circle can be constructed using the method
of Anosov and Katok \cite{ak} (see  \cite{fathe}, \cite{fayk}, \cite{fays}, \cite{ha}, \cite{he} for examples and further developments about this method).

In contrast, any continuous map of the closed annulus of degree $d$ ($|d|>1$) is semi conjugate to $z^d$ on $S^1$ (see Corollary \ref{cerrado} in Section \ref{shd}). In Section 2 we generalize the concept of rotation number to endomorphisms of the circle and the annulus and show that it is a continuous map onto $S ^1$ that semi-conjugates with $z^d$. However, this is no longer true without compactness. We construct examples of degree $d$ covering maps of the open annulus that are not semiconjugate to $z^d$ on the circle.

The results obtained in \cite{iprx} give further motivation for our study, as covering maps of the annulus arise naturally when studying surface attractors. A connected set $\Lambda$ is an {\em attracting set} of an endomorphism $f$ on a manifold $M$ if there exists a neighborhood $U$ of $\Lambda$ such that the closure of $f(U)$ is a subset of $U$ and $\Lambda = \cap_{n\geq 0} f^n (U)$. The attracting set is called {\em normal} if $f$ is a local homeomorphism of $U$ and the restriction of $f$ to $\Lambda$ is a covering of $\Lambda$.

Let $M$ be a compact surface, $f$ an endomorphism of $M$ and $\Lambda$ a normal attractor of $f$. Assume that $f$ is $d:1$ in $\Lambda$. Then, it is shown in \cite{iprx}  that the immediate basin $B_0(\Lambda)$ of $\Lambda$ is an annulus
and the restriction of $f$ to it is a $d:1$ covering map. Moreover, if $A$ is an invariant component of $B_0(\Lambda)\setminus\Lambda$, then $A$ is also an annulus and $f$ is a $d:1$ covering of $A$.
Finally, if $\Lambda$ is a hyperbolic (normal) attractor, then $\Lambda$ is homeomorphic to a circle and the restriction of $f$ to $\Lambda$ is conjugate to $z\to z^d$ in $S^1$.

Several questions arise naturally:\\

\begin{enumerate}
\item
Is the restriction of a map $f$ to the immediate basin of a normal attracting set semiconjugate to $p_d(z)=z^d$ in
$\C\setminus\{0\}$?

\item
Is the restriction of $f$ to the immediate basin of $\Lambda$ semiconjugate to $z\to z^d$ in $S^1$?

\item
Is the restriction of $f$ to $A$ semiconjugate to $m_d(z)=z^d$ in $S^1$?

\end{enumerate}

We address all of these questions in this paper. The answer to (1) is no, while the answer to (2) and (3) is yes; the proof of these statements is contained in Section \ref{basins}. The more general question whether or not any covering map of the open annulus is semiconjugate to $m_d$ has a negative answer as was already pointed out, and a counterexample is given in Section \ref{contra}.

In section 2 we find sufficient conditions for the existence of a semiconjugacy with $m_d(z)=z^d$ in $S^1$. We also give a method to construct semiconjugacies, based on the rotation number, that applies always to covering maps of the circle. This leads to a classification up to conjugacy for coverings of the circle.
At first, we tried to extend, at least in part, this classification to coverings of the annulus. The results obtained throughout this work show however, that a generalization is very difficult, even for maps having empty nonwandering set.

Within the facts proved here, those which may be more interesting, in our opinion, are:\\
The proof in section 2 of the existence of a rotation number for maps of the circle, and that this rotation induces the semiconjugacy with
$m_d$ (Proposition \ref{p1}). This was applied to give a classification of maps of the circle (Theorem \ref{class1dim}).\\
There exists a semiconjugacy between $f$ and $m_d$ if and only if $f$ has an invariant trivial connector (see Definition \ref{connector}, Corollary \ref{rep2} and Corollary \ref{connec}).\\
There exists a semiconjugacy if and only if there exists a free connector (Corollary \ref{rep1} and Corollary \ref{connec}).\\
Not every covering of the annulus is semiconjugate to $m_d$ (Section 4).\\
The restriction to a normal attractor or to a basin of a normal attractor is semiconjugate to $m_d$ (Section 4).\\
The map $p_d(z)=z^d$ defined in the punctured unit disc is not $C^0$ stable (Section 4). \\

There are few references to coverings of the annulus in the literature. The better known family of examples is given by:
$$
(x,z)\in\R\times S^1\mapsto (\lambda x+\tau(z), z^d),
$$
where $\lambda$ is a positive constant less than $1$, and $\tau:S^1\to\R$ is a continuous function. The map with $\tau=0$ was introduced by Przytycki (\cite{prz}) to give an example of a map that is Axiom A but not $\Omega$-stable. Then Tsujii (\cite{tsu}) gave some examples having invariant measures that are absolutely continuous respect to Lebesgue. In \cite{bkru} the topological aspects of the attractor are studied and examples of hyperbolic attractors with nonempty interior are given. Note that in these examples the attracting sets cannot be normal, on account of the results of \cite{iprx} cited above.

In a subsequent work about coverings of the annulus we will investigate some questions relative to the existence of periodic cycles.

\section{Sufficient conditions for the existence of semiconjugacies.}

Throughout all this section, $d$ will be an integer with $|d|>1$ and $m_d$ will be the self-covering of the circle $m_d(z)=z^d$.
In this section we will show how to construct semiconjugacies of covering maps.
The first method, an extension of the {\em rotation number} is first defined for circle maps of degree of absolute value greater than one.
In the case of circle maps, the rotation number always exist, even if the map is not a covering. This immediately induces a semiconjugacy to
$m_d(z)$. We will give some consequences to the classification of covering maps of the circle in the first subsection. In the second subsection, we will give some sufficient conditions for the existence of the rotation number for maps of the annulus and introduce the
{\em shadowing argument}. In the final part of this section we introduce the {\em repeller argument}, a condition for the existence of semiconjugacies, that in subsequent sections is also seen to be necessary.

\subsection{The rotation number.}
Let $f$ be a continuous self map of $S^1$ of degree $d$, with $|d|>1$.
The circle $S^1$ is considered here as $\{z\in \C: |z|=1\}$ and its universal covering projection as the map $\pi_0: \R\to S^1$, $\pi_0 (x)= \exp{2\pi i x}\in S^1$.

We begin by proving the following lemma, which is esentially the Shadowing Lemma for expanding maps.

\begin{lemma}
\label{semi} Let  $F:\R\to \R$ be a lift of $f$. There exists continuous map $H_F^+:\R\to {\mathbb{R}}$ and
$H_F^-:\R\to {\mathbb{R}}$ satisfying the following properties:
\begin{enumerate}
\item
$H_F^+(x+1)=H_F^+(x)+1$, $H_F^-(x+1)=H_F^-(x)-1$
\item
$H_F^\pm F=dH_F^\pm$,
\item
$|H_F^+(x)-x|$ is bounded and $|H_F^-(x)+x|$ is bounded,
\item
$H_F^\pm(x)=\pm\ \lim_{n\to \infty} \frac{F^{n}(x)}{d^{n}}$.
\end{enumerate}
\end{lemma}
\begin{proof}
Consider the spaces $\mathcal H^+$ and $\mathcal H^-$ defined by
$$
\mathcal H^\pm=\{H:\R\to \R: \ H \mbox{ is continuous and }H(x+1)=H(x) \pm 1, \ \forall x\in \R \}
$$
and endowed with the following metric:
$d(H_1,H_2)=sup_{x\in\R}\{|H_1(x)-H_2(x)|\}$.

We do the proof for $H_F^+$, the other being similar.
Note that $\mathcal H^+$ is a complete metric space. Let
$T:=T_F:\mathcal{H}^+\to\mathcal{H}^+$ be defined by
$$T(H)(x)=\frac{H(F(x))}{d}.$$  Note that $F(x+1)= F(x) + d$ as the degree of $f$ is $d$.  So, $$\frac{HF(x+1)}{d} =\frac{ H (F(x)+d)}{d}= \frac{ H (F(x))+d}{d}=\frac{HF(x)}{d} +1,$$ and therefore $T$ is a well defined operator from the space $\mathcal{H}^+$ to itself.
So, \\$d(T(H_1),T(H_2))\leq \frac{1}{|d|}d(H_1,H_2)$, implying that
$T$ is a contracting operator. Then, there exists
$H^+_F\in\mathcal{H^+}$ such that $T(H^+_F)=H^+_F$. Equivalently, $H^+_FF=dH^+_F$. This proves
(1), (2) and continuity of $H^+_F$. To see (3), note that the function $x\to H^+_F(x) - x$
defined on $\R$ is continuous and ${\mathbb{Z}}$-periodic, thus
bounded:
\begin{equation*}
H^+_F(x+k) - (x+k) = H^+ _F(x) + k -x -k = H^+_F(x) - x.
\end{equation*}

To prove (4), use the previous item to obtain that there exists a constant $C$ such that
$|H^+_F(F^{n}(x))-F^{n}(x)|<C$ for all $n\in \N$. Dividing by $d^{n}$ it comes that
$|\frac{H^+_F(F^{n}(x))}{d^{n}}-\frac{F^{n}(x)}{d^{n}}|<\frac{C}{d^{n}}$.
As $H^+_FF^{n}=d^{n}H^+_F$ , then

$|H^+_F(x)-\frac{F^{n}(x)}{d^{n}}|<\frac{C}{|d|^{n}}$ for all
$n\in\N$ and so $H^+_F(x)=\lim_{n\to\infty} \frac{F^{n}(x)}{d^{n}}$.
\end{proof}

Define the {\em rotation number} of the point $x$ and the lift $F$, denoted $\rho_F(x)$, as the limit in item (4). So this lemma proves that the rotation number exists for every continuous map $f$ of the circle with degree $|d|>1$ and is a continuous function of the point $x$. It is worth noting that the properties (1), (2) and (3) of the above lemma imply that
$\rho_F(x)+1=\rho_F(x+1)$ and $\rho_{F+J}=\rho_F+J/(d-1)$.

The condition $H(x+1)=H(x)+1$ implies that $H_F^+$ induces a continuous $h_F^+$ with $(h_F^+)_*$ an isomorphism.
Then $h^+_F$ is a semiconjugacy from $f$ to $m_d$. The same for $H(x+1)=H(x)-1$: this provides a semiconjugacy $h_F^-$.
Next result proves that these are the unique possible semiconjugacies.

First two definitions:
A map $h$ in the circle is monotone if it has a monotone lift. A self-conjugacy of $m_d$ is a homeomorphism of the circle $c$ such that $cm_d=m_dc$.

\begin{prop}
\label{p1}
Let $f$ be a degree $d$ map of the circle, where $|d|>1$.
\begin{enumerate}
\item
Every semiconjugacy $h$ from $f$ to $m_d$ is the quotient map of a fixed point of the operator $T_F$, for some lift $F$ of $f$ ($T_F$ being defined either on $\mathcal H^+$ or on $\mathcal H^-$).
\item
If $f$ is assumed to be a covering map, then every semiconjugacy is monotone.
\item
If $h_1$ and $h_2$ are semiconjugacies from $f$ to $m_d$, then there exists a unique $c$ such that $h_1=ch_2$, where $c$ is a self-conjugacy of $m_d$.
\end{enumerate}
\end{prop}
\begin{proof}
Given a lift $F$ of $f$ let $H$ be a lift of $h$ and note that $HF=F+k$ for some $k\in\Z$. Then $H$ is the fixed point of $T_{F+k}$ or $T_{F-k}$ according to the degree $\pm 1$ of $h$.

The first assertion follows because there exists a lift $H$ of a semiconjugacy $h$ that is fixed point of $T$; Moreover, as $h_*$ is an isomorphism, then $H(x+1)=H(x)\pm 1$, so $H$ belongs to one of the spaces $ \mathcal H^+$ or to $\mathcal H^-$.
To prove the second assertion, note that if $f$ is a covering, then $F$ is monotonic, and item (4) of Lemma \ref{semi} implies that the lift
$H$ of $h$ is a monotonic function; it comes that $h$ is monotonic in $S^1$.
To prove the last assertion, we first compare $h^+_F$ with $h^+_{F'}$ when $F$ and $F'$ are lifts of $f$. Note that if $F_0$ is a lift of $f$, then every other lift is $F_j=F_0+j$ where $j\in\Z$. If $HF_0=dH$ and $j\in\Z$, then a simple calculation shows that
$H_{F_0}^+=H_{F_j}^+-j/(d-1)$.
Thus $h_{F_0}^+=ch_{F_j}^+$, where $c(z)=\exp(2\pi i j/(d-1))z$, which clearly commutes with $m_d$.
To deal with the case of orientation reversing $h$, just note that $H^-_F=-H^+_F$, so $h^-_F=ch^+_F$, where $c$ is complex conjugation.
\end{proof}

At the beginning of the next subsection we will explain when this construction can be carried on for annulus maps. But first we will use the proposition above to obtain some consequences for one dimensional dynamics.

A proper arc in $S^1$ is a closed arc with nonempty interior and nonempty exterior. If $f$ is a covering map, and $h$ a semiconjugacy to $m_d$, then $h$ is monotonic, hence $h^{-1}(y)$ is either a point or a proper arc.

\begin{defi}
\label{lambdasubefe}
Let $f$ be a covering map of the circle and $h$ a semiconjugacy to $m_d$. Define $\Lambda_f$ as the set of points $x\in S^1$ such that
$h^{-1}(h(x))$ is a proper arc (equivalently, $h^{-1}(h(x))$ does not reduce to $\{x\}$). A component of $\Lambda_f$ is usually called a Plateau of $h_f$.
\end{defi}

Note that by item $(3)$ of Proposition \ref{p1}, this definition does not depend on the choice of $h$.

Some properties are:
\begin{enumerate}
\item
$f$ is conjugate to $m_d$ if and only if $\Lambda_f$ is empty.
\item
Each component of $\Lambda_f$ has nonempty interior, and this implies it has countable many components. It follows that if $I$ is a component of $\Lambda_f$, then $I$ is equal to $h^{-1}(h(x))$ for some point $x$. Moreover $f$ is injective in $I$, because $f(x)=f(y)$ with $x\neq y$ in $I$ implies $f(I)=S^1$ as $f$ is a covering. This contradicts the fact that $h|_{I}$ is constant and hence $h|_{f(I)}$ is constant.
\item
$\Lambda_f$ is completely invariant, meaning $f^{-1}(\Lambda_f)=\Lambda_f$. Indeed, the equation $hf(I)=m_dh(I)$ for an interval $I$ implies that $h$ is constant on $I$ if and only if $h$ is constant on $f(I)$. This implies that for $I$ a component of $\Lambda_f$ it holds that
$f^{n}(I)\cap f^{m}(I)\neq\emptyset$ implies $f^n(I)=f^m(I)$.
\item
Note that $h(\Lambda_f)$ is completely invariant under $m_d$, so it is dense in $S^1$, but is not the whole circle since it is countable.
\item
If $x\in I$, where $I$ is a (pre)-periodic interval of $f$, then $f^n(x)$ is asymptotic to a periodic orbit of $f$.
\end{enumerate}

A first application of the existence of the semiconjugacy describes all the transitive covering maps of the circle. A map is transitive if there exists $x\in S^1$ such that its forward orbit is dense in $S^1$.

\begin{clly} Let $f: S^1\to S^1$ be a transitive covering map of degree $|d|>1$. Then $f$ is conjugate to $m_d$.
\end{clly}

\begin{proof} Recall from property (1) above that $f$ is conjugate to $m_d$ if and only if $\Lambda_ f= \emptyset$.  Suppose that $\Lambda_ f$ contains a non-trivial interval $I$.  Let $x\in S^1$ be such that its forward orbit is dense.  Then, there exists $n\geq 0$ such that $f^n (x)\in I$ and
$m>n$ such that $f^m(x) \in I$. It follows from item (3) of the properties of $\Lambda_f$ above that $I$ is a pre-periodic interval, contradicting that the forward orbit of $x$ is dense.
\end{proof}

The nonwandering set of every covering of $f$ is defined as the set of points $x$ such that for every neighborhood $U$ of $x$ there exists an integer $n>0$ such that $f^n(U)\cap U\neq\emptyset$. The nonwandering set of $f$ is denoted by $\Omega(f)$.

\begin{clly}
If $f$ is a covering of the circle having degree $d$ with $|d|>1$, then the set of periodic points of $f$ is dense in the nonwandering set.
\end{clly}
\begin{proof} If $x$ is a nonwandering point of $f$ that is not periodic, then item (5) above implies that $x$ does not belong to the interior of $\Lambda_f$. It follows that the image under $h_f$
of a neighborhood $U$ of $x$ contains a proper arc. As the $m_d$-periodic points are dense, it follows that the interior of $h(U)$ contains a periodic point, say $p$. Then the endpoints of $h_f^{-1}(p)$ must be  periodic points of $f$, and at least one of them is contained in $U$.
\end{proof}

The last application is a classification of coverings of the circle. The Classification Theorem, says, roughly speaking, that one can construct any degree $d$ map $f$ of the circle as follows: choose some orbits of $m_d$ and open an interval in each one of them. If the orbit was wandering, nothing to do, and if the orbit was periodic, then choose a homeomorphism of the interval and insert it as the restriction of $f$ to the first return to the interval. So the classification is in terms of the following data: given a map $f$, give the degree of $f$, the set $h_f(\Lambda_f)$, and an equivalence class of homeomorphisms of the interval for each periodic orbit of $f$ that belongs to $ \Lambda_f$. The construction is a little bit complicated since $m_d$ has nontrivial self-conjugacies.

First look at the self conjugacies of $m_d$, that is, the set of
homeomorphisms of the circle that commute with $m_d$.
\begin{clly}
\label{autoconj}
Any semiconjugacy from $m_d$ to $m_d$ is a homeomorphism. The group $G_d$ of self-conjugacies of $m_d$ is isomorphic to $D_{|d-1|}$, the dihedral group.
\end{clly}
In other words, if $\{\alpha_j : 0\leq j\leq |d-1|-1\}$ denote the $|d-1|$-roots of unity, $c_j$ is the rotation of angle $\alpha_j$ and $\bar c(z)=\bar z$ is complex conjugacy,
then the set of self conjugacies of $m_d$ is the group generated by the $c_j$ and $\bar c$.
\begin{proof}
Just apply Proposition \ref{p1} for the equation $hf=m_dh$, where $f=m_d$, and use the identity as a self-conjugacy.
\end{proof}

Note that the set $h_f(\Lambda_f)$ may contain a periodic orbit; in this case $\Lambda_f$ contains a periodic interval. Assume $I$ is a component of $\Lambda_f$ such that
$f^n(I)=I$. Then the restriction of $f^n$ to $I$ is conjugate to a homeomorphism of the interval.
Let $\varphi_1:I_1\to I_1$ and $\varphi_2:I_2\to I_2$ be homeomorphisms, where $I_1$ and $I_2$ are closed intervals. Then $\varphi_1$ and $\varphi_2$ are equivalent if
they are topologically conjugate via an increasing homeomorphism. The quotient space is denoted by $\mathcal L$ and the class of a homeomorphism $\varphi$ is denoted by $[\varphi]^+$. Define also $[\varphi]^-:=[h\varphi h^{-1}]^+$, where $h$ is any decreasing homeomorphism.
Note that $\varphi_1\in[\varphi_2]^-$ iff $\varphi_1$ and $\varphi_2$ are cojugate via a decreasing homeomorphism.\\
{\bf Example.} If $f(x)=x^2$ defined on $[0,1]$, then $[f]^+\cap [f]^-=\emptyset$, but if $f(x)=x^3$ defined on $[-1,1]$, then $[f]^+=[f]^-$.\\

A proper arc $I\subset S^1$ is oriented with the same orientation of $S^1$; say positive orientation is counterclockwise orientation.
Then one may assign elements $[f^n|_{I}]^+$ and $[f^n|_{I}]^-$ in $\mathcal L$ to each $n$-periodic component $I$ of $\Lambda_f$. As $[f^n|_{I}]^{\pm}=[f^n|_{f^jI}]^{\pm}$ for every $j>0$, this correspondence depends on the
orbit of the interval and not on the choice of $I$.

Given applications $\tau_i:\Lambda_i\to\mathcal L$, where $\Lambda_i$ is a $m_d$ completely invariant nontrivial subset of $S^1$ for $i=1,2$,
say that $\tau_1$ is $\mathcal L^+$-equivalent (resp. $\mathcal L^-$) to $\tau_2$ if there exists an orientation preserving (resp. reverting) $c\in G_d$ such that $c\Lambda_1=\Lambda_2$ and $\tau_1=\tau_2 c|_{\Lambda_1}$. The $\mathcal L^\pm$-equivalence class of $\tau$ is denoted
by $[\tau]^\pm$.

Next we define the data $\mathcal D^+_f$ associated to $f$ as follows:
\begin{enumerate}
\item
An integer of absolute value greater than one, the degree $d_f$ of $f$.
\item
The class $[\tau_f]^+$, where $\tau_f:h_f(\Lambda_f)\to\mathcal L$ is defined as follows: first choose an orientation preserving semiconjugacy $h_f$ from $f$ to $m_d$, and for each $x\in h_f(\Lambda_f)\cap Per(m_d)$,
$\tau_f(x)=[f^n|_{h_f^{-1}(x)}]^+$, where $n$ is the period of $x$. When $x\in h_f(\Lambda_f)$ is not periodic, $\tau_f(x)$ is defined as the class of the identity. Note that by item (3) of Proposition \ref{p1} the class $\tau_f(x)$ does not depend on the choice of $h_f$.
\end{enumerate}

\begin{prop}
\label{positive}
If $\mathcal D^+_f=\mathcal D^+_g$, then $f$ and $g$ are conjugate.
\end{prop}
\begin{proof}
Take orientation preserving semiconjugacies $h_f$ and $h_g$ such that $h_ff=m_dh_f$ and
$h_gg=m_dh_g$, and note that by hypothesis there exists an orientation preserving homeomorphism $c\in G_d$ such that $c(h_f(\Lambda_f))=h_g(\Lambda_g)$ and $\tau_f=\tau_gc|_{h_f(\Lambda_f)}$.
The conjugacy $\psi$ between $f$ and $g$ will be obtained by solving the equation $ch_f=h_g\psi$, and of course, $\psi f=g\psi$.\\

If $x\notin \Lambda_f$, then $ch_f(x)\notin\Lambda_g$, so $h_g^{-1}ch_f(x)$ contains a unique point, define $\psi(x)$ as this unique point.
This defines $\psi$ on $S^1\setminus \Lambda_f$. Note first that if $x\notin\Lambda_f$ then $h_g\psi fx=ch_ff x=cm_dh_fx=m_dch_fx=m_dh_g\psi x=h_gg\psi x$, which implies $\psi fx=g\psi x$ because $\psi fx\notin \Lambda_g$.

Now assume that $I$ is a component of $\Lambda_f$. It is already known which interval $J$ is the image of $I$ under $\psi$, this must be the component $J$ of $\Lambda_g$ such that $h_g(J)=ch_f(I)$; of course $J$ is wandering for $g$ iff $I$ is wandering for $f$, and that $J$ is
(pre)-periodic for $g$ iff $I$ is (pre)-periodic for $f$, and the periods coincide. So choose any increasing $\psi$ from $I$ to $J$ if $I$ is wandering for $f$ and choose an increasing conjugacy from $f$ to $g$ if $I$ is periodic for $f$, and then extend it to the $f$- grand orbit of $I$ ($\cup _{n>0,\ m>0} f^{-n} (f^m(I))$) such that $g\psi=\psi f$ as follows. If $m>0$ and $x\in I$, note that $f^m$ is injective in $I$  and define $\psi(f^m(x))=g^m(\psi(x))$; if $n>0$ and $z\in f^{-n}(f^m(I))$, then note that $L=h_g^{-1}(ch_f(z))$ is a component of $\Lambda_g$ where $g^m$ is injective, so one can define $\psi(z)=g^{-m}|_L\psi f^m(z)$.
Repeat this proceeding for each grand orbit of a wandering interval. Note that the equation $h_g\psi=ch_f$ holds in all the grand orbit of $I$: $h_g\psi f^k(x))=h_gg^k \psi(x)=m_d^k(h_g\psi(x))=m_d^kch_f(x)=cm_d^kh_f(x)=ch_f(f^k(x))$. So both equations for $\psi$ still hold.\\
This defines $\psi$ in all of $S^1$. Note that it is bijective by construction and continuous because it preserves orientation.
\end{proof}

As we want a classification, it remains to consider the case where the conjugacy between $f$ and $g$ is necessarily reverting.

Define the data $\mathcal D^-_f$ associated to $f$ as follows:
\begin{enumerate}
\item
An integer of absolute value greater than one, the degree $d_f$ of $f$.
\item
The class $[\tau_f]^-$, where $\tau_f:h_f(\Lambda_f)\to\mathcal L$ is defined as follows: first choose an orientation preserving semiconjugacy $h_f$ from $f$ to $m_d$, and for each $x\in h_f(\Lambda_f)\cap Per(m_d)$,
$\tau_f(x)=[f^n|_{h_f^{-1}(x)}]^-$, where $n$ is the period of $x$. When $x\in h_f(\Lambda_f)$ is not periodic, $\tau_f(x)$ is defined as the class of the identity. Note that by item (3) of Proposition \ref{p1} the class $\tau_f(x)$ does not depend on the choice of $h_f$.
\end{enumerate}

\begin{thm}
\label{class1dim}
Two covering maps $f$ and $g$ are conjugate if and only if $\mathcal D_f^+=\mathcal D_g^+$ or $\mathcal D_f^-=\mathcal D_g^-$.
\end{thm}
\begin{proof}
Assume first that $f$ and $g$ are conjugate, and let $h$ be a homeomorphism of $S^1$ such that $hf=gh$.
It follows that each point $z$ in $S^1$ has the same number of preimages under $f$ or $g$. This implies that the degree of $f$ and $g$ are equal. Next consider orientation preserving
semiconjugacies $h_f$ and $h_g$ such that $h_ff=m_dh_f$ and $h_gg=m_dh_g$, given by Lemma \ref{semi}. Note that $h_gh$ is a semiconjugacy from
$f$ to $m_d$, so there exists $c\in G_d$ such that $h_f=ch_gh$. Then $I$ is a component of $\Lambda_f$ iff $h(I)$ is a component of $\Lambda_g$
and $h(\Lambda_f)=\Lambda_g$.\\
Note that $c$ preserves orientation iff $h$ preserves orientation.
Moreover, $h_f(\Lambda_f)=c h_g(\Lambda_g)$, so if $z\in h_g(\Lambda_g)$ and $J=h_g^{-1}(z)$, then $I:=h^{-1}(J)$ is contained in $\Lambda_f$. If $z$ is periodic for $m_d$, with period $n$, then the restriction of $g^n$ to $J$ is conjugate to the restriction of $f^n$ to $I$; moreover, as $c(z)=h_f(I)$, it follows that $\tau_f(c(z))=\tau_g(z)$. The same equation holds trivially when $z\in h_g(\Lambda_g)$ is not periodic.\\
To prove the theorem in the other direction, note that as in Proposition \ref{positive} it can be proved that $\mathcal D^-_f=\mathcal D^-_g$ implies that there exists an orientation reversing homeomorphism $\psi$ such that $\psi f=g\psi$.

\end{proof}

The space $\mathcal H$ is uncountable, so the set of conjugacy classes of coverings of the circle is also uncountable.
However, restricting to the class of maps having all its periodic points hyperbolic, we state the following question:
How many equivalence classes of coverings are there? This question related to another one:
If $f$ has a wandering interval $I$, then $h(I)$ is a non pre-periodic orbit of $m_d$. Which non pre-periodic orbit of $m_d$ is obtained as the image $h(I)$ for some $f$, $h$ and $I\subset\Lambda_f$?

Under stronger assumptions wandering intervals are forbidden, giving:

\begin{clly}
\label{cunidim1}
The number of equivalence classes of $C^2$ covering maps of the circle all of whose periodic points are hyperbolic and critical points are non-flat is countable.
\end{clly}
Note that a $C^2$ covering $f$ may have points $z$ where $f'(z)=f''(z)=0$.
\begin{proof}
The corollary follows by two strong theorems. On one hand, Ma\~n\'e proved that under these hypotheses the set of attracting periodic orbits is finite. See Chapter 4 of \cite{mms}, where it is also proved that there cannot be wandering intervals in these cases.
\end{proof}

\subsection{The shadowing argument.}\label{shd}

We have seen in the previous subsection some applications of the generalized rotation number for covering maps of the circle. The question now is to what extent this can be carried out in the annulus. Throughout this subsection, $f$ is just a degree $d$ ($|d|>1$) continuous self map of the annulus.

Let $A$ be the open annulus, $A=(0,1)\times S^1$ and $f$ a degree $d$ covering map of $A$.
We will use some standard notations: the universal covering projection is the map
$\pi:\tilde A=(0,1)\times \R\to A$ given by $\pi(x,y)=(x,\exp(2\pi i y))$.
We will denote by $F$ any lift of $f$ to the universal covering, that is, $F$ is a map
satisfying $f \pi=\pi F$. Note that $F(x,y+1)=F(x,y)+(0,d)$.
For a point $(x_0,y_0)\in \tilde A$, let $(x_n,y_n)=F^n(x_0,y_0)$, where $n\geq 0$.
The following statement can be proved exactly as Lemma \ref{semi}:
\begin{lemma}
\label{semi2} Let $f:A\to A$ be a continuous map of degree $d$, $|d|>1$, and let
$F:\tilde A\to \tilde A$ be a lift of $f$. Let $K$ be a compact $f$-invariant ($f(K)\subset K$) subset of the annulus, and
$\tilde K=\pi^{-1}(K)$. Then there exists a continuous map $H^+_F:\tilde K\to {\mathbb{R}}$ such that:

\begin{enumerate}
\item
$H_F^+(x+1,y)=H_F^+(x,y)+1$,
\item
$H^+_FF=dH^+_F$,
\item
$|H^+_F(x,y)-y|$ is bounded on $\tilde K$,
\item
$H^+_F(x)=\lim_{n\to \infty} \frac{y_n}{d^{n}}$.

\end{enumerate}
\end{lemma}

As in Lemma \ref{semi}, there is also a continuous $H_F^-$ satisfying symmetric properties.

As before, the quotient map of the function $H^+_F$, denoted $h^+_F$, is well defined in $K$ because $H^+_F(x,y+1)=H^+_F(x,y)+1$,
and satisfies $h^+_Ff|_K=m_dh^+_F$.

Note that the existence of $H$ depends in Lemma \ref{semi2} on the existence of a compact invariant set; in particular, the
{\em rotation number} of the point $(x,y)$ and lift $F$, defined as the limit in item (4) of the lemma above, does not always exist.

However, if $f$ extends to a map of the closed annulus, then:

\begin{clly}
\label{cerrado}
If $f$ is a degree $d$ continuous self map of the closed annulus and $|d|>1$, then $f$ is semiconjugate to $m_d$ on $S^1$.
\end{clly}

There is another kind of hypothesis implying the existence of a semiconjugacy.

\begin{clly}
\label{acotado}
If a lift $F$ of $f$ satisfies $\sup \{|y_1-dy_0|\}<\infty$, where $(x_0,y_0)$ is assumed to range over all of $\tilde A$, and $(x_1,y_1)$ denotes the $F$ image of $(x_0,y_0)$, then there exists a semiconjugacy $H$ satisfying properties (1) to (4) of the statement of Lemma \ref{semi2}.
\end{clly}
\begin{proof}
Define the space $\mathcal H_b$ as the set of continuous $H:\tilde A\to \R$ such that $H(x,y+1)=H(x,y)+1$ and require an additional condition:
$$
\sup \{|H(x,y)-y|\ :\ (x,y)\in \tilde A\}<\infty.
$$

Even that the functions $H$ in this space are not bounded, the metric $d(H_1,H_2)=\sup|H_1(x,y)-H_2(x,y)|$ is well defined on $\mathcal H_b$, which becomes a complete metric space.
Note that the hypothesis on $F$ implies that the operator $T(H)=\frac{1}{d}HF$ is a contraction from $\mathcal H_b$ to itself. The proof follows as that of Lemma \ref{semi}.
\end{proof}

We say that an open subset $U\subset A$ is {\it essential}, if $i_* (H_1 (U, \Z)) = \Z$, where $i_*: H_1 (U, \Z) \to H_1 (A,\Z)$ is the induced map in homology by the inclusion $i: U \to A$.  We say that a subset $X\subset A$
is {\it essential} if
any open neighbourhood of $X$ in $A$ is essential. Equivalently, a subset $X\subset A$ is essential if and only if it intersects every connector (see Definition \ref{connector} below).

\begin{lemma}
\label{sur} If $K$ is a compact essential subset of $A$, then for any lift $F$ of $f$, the map $h^+_F: K \to S^1$ is surjective.
\end{lemma}

\begin{proof}
Extend $h_F^+$ to a continuous map $h : U \to S^1$, where $U$ is a
neighbourhood $K$. Note that $h$ has degree $1$, as it coincides
with $h_F$ over $K$, and $H_F(x+1) = H_F(x) + 1$ on $\tilde K$. It follows that if $\gamma$ is a simple closed nontrivial curve in $U$, then
$h(\gamma)=S^1$, so $h$ is surjective in $U$. Take $x \in S^1$ and a sequence of neighbourhoods of $K$, $U_n \subset U$ such that $\cap U_n=K$. For all $n\geq 0$ pick $x_n\in U_n$ such that $h(x_n)=x$. Take a convergent subsequence $\lim_{k\to \infty} x_{n_k} = z \in K$. Then, by continuity $h(z) = x$, finishing the proof.
\end{proof}

It also follows that under the hypothesis of this lemma, $H^+_F$ is surjective in $\pi^{-1}(K)$.

As another application, let $U$ be an open set homeomorphic to the annulus and $f:U\to f(U)$ a continuous map of degree $d$, $|d|>1$. Assume that $\overline{f(U)}\subset U$ and let
$K=\cap _{n\geq 0}f^n(U)$. Then $f|_K$ is semi-conjugate to $m_d$.

\subsection{The repeller argument.}
\begin{defi}
\label{connector}
A subset $C$ of $A$ is called a connector if it is connected, closed, and has accumulation points in both components of the boundary of $A$.
In other words, given $r>0$ there are points $(x,y)$ and $(x',y')$ in $C$ with $x<r$ and $x'>1-r$.
A connector is called trivial if is not essential.
\end{defi}

If $f$ is a degree $d$ covering map of the annulus, then each connected component of the preimage of a connector is a connector. The preimage of a trivial connector is equal to the disjoint union of $|d|$ trivial connectors.

\begin{rk}\label{con} There exists exactly one connected component of the complement of a trivial connector that contains contains a connector. Uniqueness follows from the fact that the union of two disjoint connectors separates the annulus, and so the existence of two connectors in different components of the complement of $C$ implies that $C$ is disconnected.  To prove existence, take an open trivial neighbourhood $U$ of C.  Then, the complement of $U$ contains a (trivial) connector.
\end{rk}
When $C$ is trivial, the component of $A\setminus C$ that contains a connector is called the big component of the complement of $C$.

\begin{rk}  If $C$ is a trivial connector, then $\pi (\tilde C) = C$ for any connected component $\tilde C$ of $\pi ^{-1} (C)$.  To see this take $C'$ a trivial connector such that $C\cap C' = \emptyset$ (see remark \ref{con}).  Then, the big component of $A\backslash C'$ is a topological disc $D$ containing $C$.  Then, $\pi ^{-1} (D)$ is a disjoint union of disks $D_{\alpha}$, each one of them mapped homeomorphically onto $D$.  Then, the assertion follows.
\end{rk}

It follows that $\tilde C$ separates $\tilde C -(1,0) $ from $\tilde C +(1,0)$ for any connected component $\tilde C$ of $\pi ^{-1} (C)$.  Note that this induces an order among families of parirwise disjoint connectors in $A$ as follows. Firstly we fix a connector $C$ in the family and take  $\tilde C$ a connected component of $\pi ^{-1} (C)$.  This is the first element for the order. Let $C_1$ and $C_2$ be any two other connectors in the family.  We say that $C_1< C_2$ if the preimages $\tilde C_1$ and $\tilde C_2$ of $C_1$ and $C_2$ under $\pi$ that separate $\tilde C$ and $\tilde C + (1,0)$ verify that $\tilde C_1$ separates $\tilde C$ from $\tilde C_2$.

\begin{defi}
A connector $C$ is repelling for a map $f$ if $f(C)=C$, there exists a trivial open neighborhood $U$ of $C$ such that $f(U)$ contains the closure of $U$ and every point in $U$ has a preorbit contained in $U$ that converges to $C$.
\end{defi}

\begin{defi}
A connector $C$ is called free for a covering map $f$ if $f(C)\cap C=\emptyset$.
\end{defi}

\begin{prop}
\label{repeller}
If $f$ is a degree $d$ covering of the open annulus $A$ having a trivial free connector, then there exist $|d-1|$ repelling connectors for $f$.
\end{prop}

\begin{proof}
Begin assuming $d>0$. Let $C$ be a trivial free connector and $C_1,\ldots,C_d$ the components of $f^{-1}(C)$. Denote by $D_1,\ldots,D_d$ the components of $A\setminus f^{-1}(C)$ that contain connectors (note that $D_1\cup D_2\cup\cdots\cup D_d$ is equal to the intersection of the big components of $A\setminus C_1,\ldots,A\setminus C_d$).

Moreover, for each $1\leq i\leq d$, the set $f(D_i)$ is equal to the big component of $A\setminus C$.
On the other hand, as $C$ is free, it follows that $f^{-1}(C)\cap C=\emptyset$. Thus $C$ is contained in some $D_i$, say $D_d$, to simplify.
It follows that $f(D_i)$ contains the closure of $D_i$ for all $1\leq i\leq d-1$, but not for $i=d$.

Let $i\neq d$ and define $\tilde D_i$ as the intersection for $n\geq 0$ of the sets
$D_i^{(n)}:=\cap_{j=0}^nf^{-j}(D_i)$.
Note that each $D_i^{(n)}$ is equal to a connected component of $f^{-n}(D_i)$, this follows easily by induction.

Then $\tilde D_i$ is a repelling connector:
\begin{enumerate}
\item
It is closed in the open annulus because for every $n>0$ the closure of $D_i^{(n)}$ is contained in $D_i^{(n-1)}$.
\item
It is connected because each $D_i^{(n)}$ is connected and the sequence is decreasing.
\item $f(\tilde D_i) = \tilde D_i$.

\end{enumerate}

Next consider the case $d<-1$, and let $d_0=-d$. The precedeing argument shows that there exists at least one invariant connector $C'$.
To obtain the result, $d_0$ extra repelling connectors are needed. Let $C_i$ denote the components of the preimage of $C'$, so $f^{-1}(C')=C'\cup C_1\cup C_2\cdots\cup C_{d_0-1}$. This determines $d_0$ regions, each one of which satisfies $f(D_i)\supset \bar D_i$, excepting for two regions, say $D_1$ and $D_2$, that have $C'$ in one of its boundary components. For these two components, it is only known that
$f(D_i)$ contains $D_i$, but not its closure. However, using that $d$ is negative, it holds that if $i=1,2$ then the closure of $D_i\cap f^{-1}(D_i)$ is disjoint with $C'$, so $f(D_i)$ contains the closure of $D_i\cap f^{-1}(D_i)$. In any case ($1\leq i\leq d_0$), the intersection $\cap_{n\geq 0} f^{-n}(D_i)$ defines a repelling invariant connector, disjoint from $C'$. Thus we found $d_0+1=|d-1|$ repelling connectors.
\end{proof}

It also follows from the construction that the complement of each repeller $\tilde D_i$ is connected. This will be used in the next result.

\begin{clly}
\label{rep1}
Let $f$ be a degree $d$ covering of the open annulus and assume that there exists a trivial free connector $C$. Then there exists a semiconjugacy between $f$ and $m_d$ in $S^1$.
\end{clly}
\begin{proof}
By proposition \ref{repeller} there exist $C$ a repelling connector.  Given $n\in \N$, let $f^{-n} (C) = \{C_0 =C, C_1, \ldots, C_{d^n -1}\}$ ordered as explained above with $C_0$ as the first element and such that $C_i < C_j$ if and only if $i<j$.  Define $h(\tilde C_j)=exp(2\pi i j/d^n)$ .

It follows that $h$ is defined on the set $\cup _{i=0}^\infty f ^{-n} (C)$.  This function is uniquely extended to the set $X$ equal to the closure of

$\cup _{i=0}^\infty f ^{-n} (C)$.  Not that, by construction,  if $U$ is a connected component of $A\backslash X$, then $h$ is constant in $\partial U$. Defining $h$ on $U$ as this constant, we have a continuous map $h$ defined on $A$.  By construction, $h$ is a semiconjugacy from $f$ to $m_d$ on $X$ and then on $A$.

%
%
%
\end{proof}

\begin{clly}
\label{rep2}
Let $f$ be a degree $d$ covering of the open annulus and assume that there exists a trivial connector $C$ such that $f(C)=C$. Then there exists a semiconjugacy between $f$ and $m_d$.
\end{clly}
\begin{proof} The preimage of a trivial connector $C$ has exactly $|d|$ components, each one of which is a trivial connector. So, if $C$ is invariant, then $f^{-1}(C)$ is equal to the union of $C$ plus $|d|-1$ connectors $C_1,\ldots,C_{|d|-1}$. Each of the $C_i$ is a trivial free connector, hence the previous corollary applies directly.
\end{proof}

Here we have an application:

\noindent
{\bf Example.} Let $f(x,z)=(\varphi(x),z^2\exp(2\pi i/(1-x)))$, where $\varphi(x)$ is a homeomorphism of $[0,1]$ such that every point has $\omega$-limit set equal to $1$.
Note that $f$ cannot be extended to the closed annulus.  However, using Corollary \ref{rep2}, there exists a semiconjugacy between $f$ and $m_2$.  Indeed, take $p\in A$
and a simple arc $\gamma$ joining $p$ and $f(p)$ such that $\gamma(t) = (x(t), z(t))$ with $x(t)$ an increasing function.  So, $C=\cup_{n\in \Z}\gamma _n$ is an invariant
connector, where $\gamma_n = f^n (\gamma)$ for $n\geq 0$, and for $n<0$, $\gamma _n$ is defined by induction, beginning with $\gamma_{-1}$ being the lift of
$\gamma(1-t)$ starting at $x$.  Note that $C$ is trivial, due to the choice of $\gamma$ with increasing $x(t)$ and the formula for $f$.

\section{Necessary conditions.}

Throughout this section we will assume that there exists a semiconjugacy $h$ between a covering map $f$ of degree $d$ ($d>1$) of the annulus
and $m_d(z)=z^d$ in $S^1$.  It is as always assumed that the map $h_*$ induced by $h$ on homotopy is an isomorphism.

\begin{lemma}
\label{l1.1}
Let $p$ and $q$ be fixed points of $f$. The following conditions are equivalent:
\begin{enumerate}
\item
There exists a curve $\gamma$ from $p$ to $q$ such that $\gamma_n(1)=q$ for every $n>0$, where $\gamma_n$ is the unique lift of $\gamma$
by $f^n$ that begins at $p$.
\item
There exists a curve $\gamma$ from $p$ to $q$ such that $f\gamma\sim \gamma$ where $\sim$ means homotopic relative to endpoints.
\item
$h(p)=h(q)$.
\end{enumerate}
\end{lemma}
\begin{proof}
(1) implies (2): Let $\gamma$ be as in (1). If $f(\gamma)$ is not homotopic to $\gamma$ then $\beta=\gamma^{-1}.f(\gamma)$ is not null-homotopic.
By hypothesis, the lift of $\beta$ under $f^n$ is the closed curve $\gamma_n^{-1}.\gamma_{n+1}$
for every $n\geq 0$; then the curve $\beta\neq 0$ belongs to $\cap_{j\geq 0} f_*^j(\Z)$, where $f_*$ is the induced map in homology. This cannot hold since $f_*$ is multiplication by $d$.\\
(2) implies (1):
If $f(\gamma)$ is homotopic to $\gamma$, then the final point of $\gamma_1$ is $q$, because $\gamma_1^{-1}.\gamma$ is the lift of the null- homotopic curve
$\gamma^{-1}.f(\gamma)$. It follows also that $\gamma_1^{-1}.\gamma$ is null-homotopic; hence, as $\gamma_2$ is the lift of of $\gamma_1$ that begins at $p$ the same argument shows that $\gamma_2^{-1}.\gamma_1$ is null-homotopic. Then proceed by induction.\\
(3) implies (2): Let $\tilde p$ be a lift of $p$ and  $F$ a lift of $f$ that fixes $\tilde p$. If $H$ is any lift of $h$, then there exists $k\in \Z$ such that $HF= dH +k$. As $F(\tilde p ) = \tilde p$, one gets $H(\tilde p) = d H(\tilde p) + k$ and so $H(\tilde p ) (1-d) = k$. That is, $H(\tilde p )= \frac{k}{1-d}$. As $h(p) = h(q)$, if $\tilde q$ is a lift of $q$, then there exists $l\in \Z$ such that $H(\tilde q) =\frac{k}{1-d} +l$, and so $H (\tilde q -l) = \frac{k}{1-d}$. We want to prove now that $F(\tilde q -l) = \tilde q -l$, because then any arc joining $\tilde p$ and $\tilde q -l$ projects to an arc joining $p$ and $q$ which is homotopic to its image, proving (2).

We know that there exists $j\in \Z$ such that $F(\tilde q) = \tilde q +j$, as $q$ is fixed by $f$. Then, $F(\tilde q -l)= \tilde q +j -dl$ and $HF(\tilde q-l) = \frac{k}{1-d}+j +l (1-d) = dH(\tilde q -l)+k = d \frac{k}{1-d}+k$. So $d \frac{k}{1-d}+k = \frac{k}{1-d}+j +l (1-d)$, $(d-1) \frac{k}{1-d}+k = j +l (1-d)$, $j= -l (1-d)$. This gives $F(\tilde q-l) = \tilde q-l (1-d)-dl = \tilde q-l$ as wanted.

(2) implies (3): Let $\gamma$ be as in (2). Then, $h(\gamma)\sim hf(\gamma)= m_dh(\gamma)$.  But $h(\gamma)$ joins $h(p)$ to $h(q)$, fixed points of $m_d$.  So,
$h(\gamma)\sim m_dh(\gamma)$ implies $h(p) = h(q)$. To see this, note that any lift of the map $m_d$ has exactly one fixed point, which is equivalent to the fact that if $h(p)$ and $h(q)$ are different fixed points of $m_d$ any arc connecting them is not homotopic to its image by $m_d$.
\end{proof}

\subsection{Connectivity.}

Denote by $A^*$ the compactification of the annulus with two points, that is $A^*=A\cup\{N,S\}$. Considered with its usual topology, it is
homeomorphic to the two-sphere. We will use the spherical metric in $A^*$.

\begin{lemma}
\label{conn}
Let $h$ be a semiconjugacy between $f$ and $m_d$. Then $h^{-1}(z)$ contains a connector, for every $z\in S^1$.
\end{lemma}
\begin{proof}
Note that $h^{-1}(z)$ is a closed subset of the open annulus. Proving that the compact set $K=h^{-1}(z)\cup\{N,S\}\subset A^*$ contains a connector (that is, has a component containing $N$ and $S$) is equivalent to the thesis of the lemma.
By compactness of $K$, there exists a sequence $\mathcal U_n$ of finite coverings of $K$ satisfying the following properties:
\begin{enumerate}

\item
Each element of $\mathcal U_n$ is an open disc.
\item
$\mathcal U_{n+1}$ is a refinement of $\mathcal U_n$.
\item
The distance from $K$ to the complement of the union of the elements of $\mathcal U_n$ is less than $1/n$.

\end{enumerate}

If $K_n$ denotes the closure of the union of the elements of $\mathcal U_n$, then the distance from $K$ to the complement of $K_n$ is less than $1/n$, so the intersection of the $K_n$ is equal to $K$. Moreover, each $K_n$ is the union of a finite number of closed discs.

It is claimed claim now that $K$ contains a connector if and only if any $K_n$ does. One implication is trivial; to prove the other, assume that every $K_n$ contains a connector. As each $K_n$ has finitely many components, there exists a nested sequence of connectors $C_n\subset K_n$. Then $\cap _n C_n$ is a connector contained in $K$.

Therefore, if $K$ does not contain a connector, then some $K_n$ does not contain a connector, and as $K_n$ is a finite union of closed discs, it is easy to find a closed nontrivial curve $\gamma$ in the complement of $K_n$, thus in the complement of $K$. This is a contradiction since $h(\gamma)=S^1$ for every nontrivial $\gamma$ because $h_*$ is an isomorphism.
\end{proof}

The following result completed with Corollary \ref{rep2} implies that $f$ is semiconjugate to $m_d$ if and only if there exists a trivial invariant connector.

\begin{clly}
\label{connec}
With $f$ as above, there exists an invariant trivial connector contained in $h^{-1}(1)$.
\end{clly}
\begin{proof}
Use the previous lemma to obtain a component $C$ of $h^{-1}(1)$ that is a connector. Note that
$C$ is trivial because it is disjoint to a connector in $h^{-1}(z)$ for every $z$, $z \neq 1$. If $C$ is not free, then it is invariant ($f(C)\subset C$) as $fh^{-1}(1)\subset h^{-1}(1)$ , and if it is free, then the repeller argument (section \ref{repeller}) provides an invariant connector.
\end{proof}

We will use the following two propositions to construct an example of a $d:1$ covering of the annulus that is not semi-conjugate to $m_d$.
Note that any lift $H$ of the semiconjugacy $h$ satisfies $H(x,y+1)=H(x,y)\pm 1$, since $h_*$ is an isomorphism.

\begin{prop}
\label{liftsemi}
Given a compact set $L\subset (0,1)$ there exists a constant $M$
such that $|H(x,y)-y|\leq M$ for every $x\in L$ and $y\in \R$ whenever $H$ is a lift of $h$ such that $H(x,y+1)=H(x,y)+ 1$.
\end{prop}
\begin{proof} There exists a constant $M$ such that $|H(x,y)-y|\leq M$ for every $(x,y)\in L\times [0,1]$; but $H(x,y+1)=H(x,y)+1$ implies that the same constant $M$ bounds $|H(x,y)-y|$ for $x\in L$ and $y\in \R$.
\end{proof}
When $H(x,y+1)=H(x,y)- 1$ the conclusion is changed to $|H(x,y)+y|\leq M$.
\subsection{Bounded preimages.}

We note $\gamma _1 \wedge \gamma _ 2$ the algebraic intersection number between two arcs in $A$ whenever it is defined.  In particular, when both arcs are loops, when one of the arcs is proper and the other is a loop or when both
arcs are defined on compact intervals but the endpoints of any of the arcs does not belong to the other arc. For convention, we set $c \wedge \gamma = 1 $ if $c: (0,1) \to A$ and $\gamma : [0,1] \to A$ verify:

$$c (t) = (t,1), \ \gamma (t) = (1/2, e^{2\pi i t}).$$

Note that the arc $c$ is a connector whose intersection number with any loop in $A$ gives the homology class of the loop.

If $\alpha$ is a loop in $A$, denote by $j\alpha$ the concatenation of $\alpha$ with itself $j$ times.

\begin{prop}
\label{condition}
Let $f$ be a covering of the open annulus $A$ and assume that $f$ is semiconjugate to $m_d$. Then the following condition holds:\\
(*) For each compact set $K\subset A$ there exists a number $C_K$ such that: given $\alpha\subset A$ a simple closed curve, $n\geq 1$ and $j\in [1, \ldots, d^{n-1}]$ then any $f^n-$ lift $\beta$ of $j \alpha$ with endpoints in $K$ satisfies $|\beta \wedge c|\leq C_K$ whenever it is defined.
\end{prop}

\begin{proof}
Let $h$ be the semiconjugacy between $f$ and $m_d$ and let $H$ and $F$ be lifts of $h$ and $f$ verifying $HF=dH$ Assume that
$H(x,y+1)=H(x,y)+1$. Take $a,b\in (0,1)$  such that the set $\tilde K=[a,b]\times \R$ contains $\pi ^{-1} (K)$.
By Proposition \ref{liftsemi} above, there exists a constant $M$ such that $|H(x,y)-y|\leq M$ whenever $(x,y)\in \tilde K$.

Take $\alpha$, $n$, $j$ and $\beta$ as in the statement. Let $\tilde\beta$ be a lift of
$\beta$ to the universal covering. As the endpoints of $\beta$ belong to $K$, then the extreme points $(x_1,y_1)$ and $(x_2,y_2)$ of $\tilde \beta$ belong to $\tilde K$.
Note that it is enough to show that $|y_2-y_1|$ is bounded by a constant $C_K$. We will prove that this holds with $C_K=2M+1$.

Note that $F^n(x_1,y_1)$ and $F^n(x_2,y_2)$ are the endpoints of a lift of $j\alpha$ to the universal covering. This means that $|F^n(x_1,y_1)-F^n(x_2,y_2)|=(0,j)$. It follows
that $|H (F^n(x_1,y_1))-H(F^n(x_2,y_2))|=j$.

Then,
$$
|d^nH(x_1,y_1)-d^nH(x_2,y_2)|=|H(F^n(x_1,y_1))-H(F^n(x_2,y_2))|=j\leq d^n,
$$
so $|H(x_1,y_1)-H(x_2,y_2)|\leq 1$.
Finally, using that the endpoints of $\beta$ belong to $K$, it follows that
$$
|y_1-y_2|\leq |y_1-H(x_1,y_1)|+|H(x_1,y_1)-H(x_2,y_2)|+|H(x_2,y_2)-y_2|\leq 2M+1.
$$

We assumed at the beginning that $H(x,y+1)=H(x,y)+1$, for the other possibility the proof is similar.

\end{proof}

\subsection{Counterexample.}\label{contra}
Now we construct $f$, a covering of the open annulus for which the condition (*) introduced in Proposition \ref{condition} does not hold.
This implies that $f$ is not semiconjugate to $m_d$.

Let $\{a_n :\ n\in \Z\}$ be an increasing sequence of positive real numbers such that $a_n\to 0$ when $n\to-\infty$ and $a_n\to 1$ when $n\to+\infty$.
Define the annuli $A_n$ as the product $[a_n,a_{n+1}]\times S^1$, for each $n\in \Z$.
Let also $\lambda_n$ be the affine increasing homeomorphism carrying $[0,1]$ onto $[a_n,a_{n+1}]$.
Define $f(x,z)=(\lambda_{n+1}(\lambda_n^{-1}(x)),z^2)$ for $x\in [a_n,a_{n+1}]$, for every $n\leq -1$, that is, $(x,z)\in \cup_{n<0}A_n$.

Assume $f$ constructed until the annulus $A_{n-2}$ for some $n\geq 1$ and we will show how to construct the restriction of $f$ to $A_{n-1}$.
We will suppose that $f(a_k,z)=(a_{k+1},z^2)$ for every $k\leq n-1$ and every $z\in S^1$, $f(A_{k-1})= A_k$.

Let $\alpha$ be a curve in $A_0$ such that
\begin{enumerate}
\item
$\alpha$ joins $(a_0,1)$ with $(a_1,1)$.
\item
The lift $\alpha_0$ of $\alpha$ to the universal covering that begins at $(a_0,0)$, ends at $(a_1,n)$.
\item
$\beta:=f^{n-1}(\alpha)$ is simple.
\end{enumerate}

Note that $f^{n-1}$ is already defined in $A_0$. To prove that such an $\alpha$ exists, take first any $\alpha'$ satisfying the first and second conditions. Then $f^{n-1}(\alpha')$ is a curve joining $(a_{n-1},1)$ with $(a_n,1)$. Maybe $f^{n-1}(\alpha')$ is not simple, but there exists a simple curve $\beta$ homotopic to $f^{n-1}(\alpha')$ and with the same extreme points. Then define $\alpha$ as the lift of $\beta$ under $f^{n-1}$ that begins at the point $(a_0,1)$.

Choose any simple arc $\beta'$ disjoint from $\beta$ and contained in $A_{n-1}$, joining the points $(a_{n-1},-1)$ and $(a_n,-1)$.
Note that $f^{-(n-1)}(\beta')$ is the union of $2^{n-1}$ curves, all of them disjoint from $\alpha$. Choose any one of these curves and denote it by
$\alpha'$. Note that it does not intersect $\alpha$.
Observe that there is a lift $\alpha'_0$ of $\alpha'$ that begins in a point $(a_0,t)$ and ends at $(a_1,n+t)$ in the universal covering,$t>0$. Then choose a point $Y\in\alpha'$ whose lift $Y'$ in $\alpha'_0$ has second coordinate greater than $n$. Also choose a point $X$ in $\alpha$ whose lift $X'$ in $\alpha_0$ has second coordinate less that $1/2$.

Observe that $f^{n-1}(X)\in\beta$ and $f^{n-1}(Y)\in\beta'$.

The complement of $\beta\cup\beta'$ in the interior of $A_{n-1}$ consists of two open discs, and each one of them is homeomorphic to the complement of $s$ in the interior of $A_n$, where $s$ is the segment $\{(x,1)\ :\ a_n<x<a_{n+1}\}.$
Then it is possible to take a homeomorphism from each of these components and extend it to the boundary in such a way that the image of $\beta$ is $s$ and the image of $\beta'$ is also $s$, and carrying $f^{N-1}(X)$ and $f^{N-1}(Y)$ to the same point $p\in s$. If the homeomorphisms are taken carefully, they induce a covering $f$ from $A_{n-1}$ to $A_n$.
Now take a simple essential closed curve $\gamma$ contained in $A_n$ and with base point $p$. Note that for some $j\in [1, \ldots, 2^{n-1}]$, the curve $j\gamma$ lifts under $f^n$ to a curve joining $X$ to $Y$. But the difference between the second coordinates of $Y'$ and $X'$ is greater than $n-1$.
By the remark preceding Proposition \ref{condition}, it follows that the intersection number of a lift of $j\gamma$ and the connector $c$ in $A_0$ exceeds $n-1$. Taking $K=A_0$ in Proposition \ref{condition}, note that $C_K\geq n-1$, and as this can be done for every positive $n$, it follows that $f$ does not satisfy condition (*).

\subsection{Inverse limit}

Here we will define the inverse limit of a covering $f$ and prove that if $f$ has a fundamental domain, then its inverse limit is semiconjugate to the inverse limit of $m_d$. This shows that for the example given above, even that $f$ is not semiconjugate to $m_d$, the semiconjugacy can be defined on inverse limits.

\begin{defi}
\label{il}
The inverse limit set of a self map $f$ of a topological space $X$ (denoted $X_f$) is defined as the set of orbits of points in $X$, endowed with the product topology inherited from the countable product of $X$. If $\sigma_f$ denotes the shift map on $X_f$, then $\sigma_f$ is a homeomorphism onto $X_f$ and if $n\in\Z$, then $\pi_n\sigma_f=f\pi_n$, where $\pi_n:X_f\to X$ denotes the projection onto the $n^{th}$ coordinate.
The inverse limit set of the map $m_d$ in $S^1$ is denoted by $S_d$; the map $\sigma_d:=\sigma_{m_d}$ is commonly known as a solenoid.
\end{defi}

\begin{defi}
\label{fd}
Let $f$ be a covering of the annulus $A$. A fundamental domain for $f$ is a compact essential annulus $A_0$ such that the
following conditions hold:
\begin{enumerate}

\item
every orbit of $f$ hits at least once and at most twice in $A_0$, and
\item
$f(A_0)$ does not intersect the interior of $A_0$.
\end{enumerate}
\end{defi}

The fact that $f$ is a local homeomorphism implies that the preimage of an essential annulus is an essential annulus (note that the preimage of a nontrivial simple closed curve is a simple closed curve, and the preimage of a trivial closed curve is the union of $|d|$ disjoint simple closed curves). Then, for every $n>0$, $A_{-n}=f^{-n}(A_0)$ is an annulus, whose interior is disjoint to $A_0$.

Note that the nonwandering set of $f$ in $A$ is empty. To prove this, assume that $x_0$ is a nonwandering point. Then there exists an $f$-orbit of $x_0$, denoted $\{x_n\}_{n\in\Z}$, contained in $\Omega(f)$. It follows there exists at least one $k\in\Z$ such that $x_k\in A_0$.
But a point in $A_0$ cannot be nonwandering.

Consider $A^*$ as the compactification of $A$ with two points $N$ and $S$.
Note that $f$ extends to $A^*$ fixing the boundary points (otherwise it would have nonwandering points in $A$).

For a simple nontrivial closed curve $L\subset A$ define $S_L$
(resp. $N_L$) as the component of $A^*\setminus L$ that contains $S$ (resp. $N$).

If $L$ denotes a nontrivial simple closed curve in the interior of $A_0$ then $f^{-1}(L)$ is a nontrivial simple closed curve not intersecting
$L$. It follows that $f^{-1}(L)$ is a subset either of $S_{L}$ or of $N_L$. Assume that
$f^{-1}(L)\subset S_L$. It follows that $f^{-n}(L)\subset S_{f^{-n+1}(L)}$ for every $n>0$. Then $\cup_{n\geq 0} A_{-n}$ is equal to $S_{L_0}$, where $L_0$ denotes the northern boundary of $A_0$.
Obviously $L_{-1}=f^{-1}(L_0)$ is the southern boundary of $A_0$.

The preimage of a fundamental domain is also a fundamental domain, but not the image. For instance, take $p_d(z)=z^d$ in the punctured unit disc, $\D\setminus\{0\}$, let $\gamma$ be a simple closed curve (close to a circle) centered at the origin but that is not symmetric with respect to the origin, this means there exists a point $x\in\gamma$ such that $-x\notin\gamma$. Then $p_d^{-1}(\gamma)\cap\gamma=\emptyset$, and $p_d(\gamma)$ is not a simple curve. It follows that the set of points between $\gamma$ and $p_d^{-1}(\gamma)$ is a fundamental domain but its image is not.

\begin{prop}
\label{inverse}
If a covering map $f$ of the annulus has a fundamental domain, then $\sigma_f$ is semiconjugate to $\sigma_d$, with $d$ the degree of $f$.
\end{prop}
Here, a semiconjugacy is a continuous surjective map $h:A_f\to S_d$ such that $h\sigma_f=\sigma_d h$.
\begin{proof}
Let $A_0$ be a fundamental domain for $f$. Then $\sigma_f$ has a fundamental domain $\bar{A_0}$ defined as the set of points $\bar z$ such that $z_0\in A_0$. However, the map $\sigma_d$ does not have a fundamental domain; note, for instance, that its nonwandering set is the whole $S_d$.
Then the image under a semiconjugacy of the fundamental domain of $\sigma_f$ must be the whole $S_d$. It suffices to construct the semiconjugacy $h$ in $\bar{A_0}$. One has to define functions $h_n:\bar{A_0}\to S^1$ to determine $h=\{h_n\}$.
It begins with the choice of a surjective map from $A_0$ to $S^1$, the function $h_0$ will depend just on the $0$-coordinate of a point
$\bar z\in\bar{A_0}$. It is obvious how to define $h_n$ for $n$ positive, but for negative $n$, one has to determine regions where $f^n$ is injective.

We proceed to do this.  Let $L_0$ and $L_{-1}$ be the boundary components of $A_0$, where $f^{-1}(L_0) = L_{-1}$.
Take any simple arc $\gamma_0$ contained in the closure of $A_0$ joining a point $x\in L_{-1}$ with the point $f(x)\in L_0$. The curve $\gamma_0$ is a connector of $A_0$. Then $f^{-1}(\gamma_0)$ is equal to the union of $d$ simple arcs, each one of which is a connector of $A_{-1}=f^{-1}(A_0)$. We assume that these arcs are enumerated as $\gamma^1_0,\ldots,\gamma^1_{d-1}$, in such a way that $\gamma^1_0$ has an extreme point in $x$, and the extreme points of $\gamma^1_j$ in $L_{-1}$ are counterclockwise ordered. By induction, we can define, for each positive $n$, a sequence $\{\gamma^n_j\ : \ 0\leq j\leq d^n-1\}$ of connectors of $A_{-n}$, such that
$\gamma^n_0$ has an extreme point in common with $\gamma^{n-1}_0$.
Moreover, if the extreme point of $\gamma^n_j$ in $f^{-n}(L_0)$ is denoted by $x_j^n$, then these points are counterclockwise oriented in the curve $f^{-n}(L_0)$. Note that the restriction of $f^n$ to the open region $D^n_j$ contained in $f^{-n}(A_0)$ and bounded by $\gamma_j^n$ and $\gamma_{j+1}^n$ is injective, for $0\leq j\leq d^n-1$.
By convenience, denote $\gamma_{d^n}^n:=\gamma_0^n$.

First define a function $\phi:A_0\to S^1$ satisfying some conditions. The function $\phi$ will be used to define $h_0$; indeed, $h_0(\bar z)$ will depend only on the value of $z_0$ and $\phi (z_0)$ will be equal to $h_0 (\bar z)$ if $z_0\in \overline A_0$. The conditions imposed on $\phi$ are: 1. $\phi$ is continuous, 2. $\phi^{-1}(1)= \gamma _0$.
3. $\phi$ carries the circle $L_{-1}$ homeomorphically onto $S^1$ in such a way that $f(y)=f(y')$ if and only if $m_d(\phi(y))=m_d(\phi(y'))$.
The last assertion allows to define $\phi(f(y))=m_d(\phi(y))$ whenever $y\in L_{-1}$.

To define $h$, begin with a point $\bar z=\{z_n\}_{n\in\Z}$ contained in $A_f$, and assume that $z_0\in A_0\setminus L_{-1}$. Assume also that $z_0\notin \gamma_0$. Then, for each positive $n$, the point $z_{-n}$ belongs to $D^n_j$ for some (unique) $j$, $0\leq j\leq d^n-1$.  Then let $h_{-n}(\bar z)$ be the unique $m_d^n$-preimage of the point $\phi(z_0)$ that belongs to the arc in $S^1$ with extreme points
$\exp(2\pi i j/d^n)$ and
$\exp(2\pi i (j+1)/d^n)$.

Now, if $z_0$ belongs to $\gamma_0$, then $z_{-n}$ belongs to some $\gamma^n_j$. Then define $h_{-n}(\bar z)=\exp(2\pi i j/d^n)$. For positive $n$, define $h_n(\bar z)=m_d^n(\phi(z_0))$.

The equation $\pi_n(h(\bar z))=h_n(\bar z)$ defines a map from $\pi^{-1}(A_0)\cap A_f$ to $S_d$. By construction, each $h_n$ is continuous, from which it follows that $h$ is continuous. Note also that the image of this map is all $S_d$.

This concludes the definition of the restriction of $h$ to the fundamental domain of $\sigma_f$.

To define $h$ in the whole $A_f$, take any $\bar z\in A_f$, Then there exists a unique $k\in\Z$ such that
$z_k\in A_0\setminus L_{-1}$. Define $h(\bar z)=\sigma_d^k(h(\sigma_f^{-k}(\bar z))$.

\end{proof}

\section{Basins.}
\label{basins}
Let $f$ be an endomorphism of a surface $M$ having an attracting set $\Lambda$ which is normal (see definition in the introduction) and has degree $d>1$; assuming that $f$ has no critical points in the basin $B_0(\Lambda )$ of $\Lambda$, it comes that the restriction of $f$ to the immediate basin of $\Lambda$ is a covering of the same degree. Moreover, if $C$ is a component of $B_0(\Lambda )\setminus \Lambda$, then $C$ is an annulus and if $C$ is invariant, then $f$ is a covering of $C$ (see \cite{iprx}). We will prove here that under these conditions, $f$ in $C$ is semiconjugate to $m_d$. We note that this is not an immediate consequence of Corollary \ref{cerrado} because the closure of $C$ is not necessarily a closed annulus.

The following example is illustrating on the situation. The map $f(z)=z^2-1$ is a hyperbolic map of the two-sphere having a superattractor at $\infty$. The restriction of $f$ to the basin of $\infty$ is a degree two covering map of the annulus $\C\setminus \hat J$ conjugate to $z\to z^2$ restricted to the exterior of the unit circle, where $J$ is the Julia set of $f$ and $\hat J$ is the filled Julia set. The restriction of $f$ to the Julia set $J$ (that is the boundary
of the basin of $\infty$), is also a covering of degree two, but this map is not semiconjugate to $m_2$. The Julia set $J$ is a curve, but is
not a simple curve. Moreover, there exists a circle contained in $J$, that is tangent to its preimage, that is also a circle: this prevents the existence of a semiconjugacy to $m_2$ on the circle.

The results mentioned above, proved in \cite{iprx}, imply that the Julia set of $f$ cannot be the attracting set of a continuous map of the sphere. Indeed, if $\Lambda$ is a connected attracting set, then there exists a basis $\mathcal B$ of neighborhoods of $\Lambda$, each homeomorphic to an annulus. Moreover each $U\in \mathcal B$ satisfies that the closure of $f(U)$ is contained in $U$.

\begin{thm}
\label{inU}
If $\Lambda$ is a normal attractor and $f$ has no critical points in the closure of the basin of $\Lambda$, then:\\
(a) the restriction of $f$ to $\Lambda$ is semiconjugate to $m_d$.\\
(b) the restriction of $f$ to $B_0(\Lambda)$ is also semiconjugate to $m_d$.
\end{thm}
\begin{proof}
The first item was proved as an application at the end of subsection 2.2.
To prove the second one, let $U$ be an annular neighborhood of $\Lambda$ such that the closure of $f(U)$ is contained in $U$.
It is known that the union for $n>0$ of the sets $f^{-n}(U)\cap B_0(\Lambda)$ equals the immediate basin $B_0(\Lambda)$ and is
an open annulus, restricted to which $f$ is a covering of degree $d$, with $|d|>1$ (see \cite{iprx} Thm. 2.). By Lemma \ref{semi2} (taking $K=\bar U$) there exists a continuous
$h:\bar U\to S^1$ such that $hf=m_dh$.  If $(\tilde A,\pi)$ denotes the universal covering of $B_0(\Lambda)$, then there exist a lift
$H:\pi^{-1}(\bar U)\to \R$ and a lift $F$ of $f$ such that $HF|_{\pi^{-1}(\bar U)}=dH$ and $H(x,y+1)=H(x,y)+1$. Note that $F$ is a homeomorphism of $\tilde A$ and that $\pi^{-1}(\bar U)$ is $F$-invariant, so there exists a unique continuous extension $H'$ of $H$
to $\tilde A$ that satisfies $H'F=dH'$.
Therefore the quotient map of $H$ is a semiconjugacy $h$ between $f$ and $m_d$. Finally, restrict $h$ to the component $C$.

\end{proof}

Two homeomorphisms may have homeomorphic basins without being conjugate. However, when restricted to the trivial dynamics in
$B\setminus\Lambda$ (where $\Lambda$ is the attractor and $B$ the basin), the fact that two fundamental domains are homeomorphic implies that the maps are conjugate.
This is not true for coverings in general.
We will consider the map $p_d(z)=z^d$ as acting in $\D^*=\D\setminus\{0\}$ (we use the notation $m_d$ for $z\to z^d$ acting on $S^1$).\\

\noindent
{\bf Example 1.} There exists a covering $f:[0,1]\times S^1\to [0,1]\times S^1$, admitting a fundamental domain, but whose restriction to
$(0,1)\times S^1$ is not conjugate to $p_d$.\\
Given any degree $d$ covering $g$ of $S^1$ that is not conjugate to $m_d$, the map $f:[0,1]\times S^1\to [0,1]\times S^1$ given by $f(x,z)=(x^2, g(z))$ is not conjugate to $p_d$.
As $g$ is not conjugate to $m_d$, there exists a periodic or wandering arc $(a,b)\subset S^1$.
Let $\Delta =\{(x,z)\in  (0,1)\times S^1: \ z\in(a,b)\}$ and $B$ a closed disc contained in $\Delta$. Assume there exists a conjugacy $H$ between $f$ and $p_{d}$. Note that $f^n(B)$ is a disc for every $n>0$, because $f^n$ is injective in $\Delta$. On the other hand, $H(B)$ is a disc and so $p_d^n$ is not injective in $H(B)$ for every large $n$.
This is a contradiction.\\

\noindent
{\bf Example 2.} There exists a map $f:[0,1]\times S^1\to [0,1]\times S^1$ with fundamental domain, but it is not semiconjugate to $p_d$.\\
Note that $f$ is defined in the closed annulus which implies that it is semiconjugate to $m_d$.\\
Let $f(x,z)=(\phi(x,z),z^d)$, where $\phi$ will be determined. The condition to be imposed on $f$ is the following: there exists a point $P$ such that the union for $n>0$ of the sets $f^{-n}(f^n(P))$ is dense in an essential annulus $A_0\subset A$.
Assume that there exists a semiconjugacy $h$ between $f$ and $p_d$. This means that $hf=p_dh$, that $h$ is surjective and is an isomorphism on first homotopy group.
It follows that $h(A_0)$ must be a nontrivial circle. If, in addition, the annulus $A_0$ is a fundamental domain for $f$, then the range
of $h$ will be a countable union of circles, hence $h$ is not surjective, a contradiction.

Let $\{a_n :\ n\in \Z\}$ be an increasing sequence of positive real numbers such that $a_n\to 0$ when $n\to-\infty$ and $a_n\to 1$ when $n\to+\infty$.
Define the annuli $A_n$ as the product $[a_n,a_{n+1}]\times S^1$, for each $n\in \Z$.
Let also $\lambda_n$ be the linear increasing homeomorphism carrying $[0,1]$ onto $[a_n,a_{n+1}]$.

The construction of $\phi$ depends on two sequences. First let $\bar z=\{z_n:n<0\}$ be a preorbit of $1$ under $m_d$, that is, $m_d(z_n)=z_{n+1}$ for $n<-1$, and $m_d(z_{-1})=1$, and assume also that $z_{-1}\neq 1$, which implies that the $z_n$ are all different. Then take an element $\bar\nu=\{\nu_n\}_{n<0}\in (0,1)^\N$.

By appropriately choosing the function $\phi$, it will come that the map $f$ will be such that, for some $P\in A_0$,
the set $f^{-n}(f^n(P))$ contains the point $(\lambda_0(\nu_{-n}),z_{-n})$ for every $n>0$.

It is clear that the sequences $\bar\nu$ and $\bar z$ can be chosen in order to make the set $\cup_{n>0}f^{-n}(f^n(P))$ dense in $A_0$, which is a fundamental domain for $f$ in $A$.

Fix a point $P=(\lambda_0(1/2),1)$ in the annulus $A_0$. To define $\phi$ we will use the sequences $\bar z$ and $\bar\nu$. The definition of $\phi$ is by induction beginning in the annulus $A_0$. Note that $\phi$ must carry the annulus $A_{n}$ into the segment $[a_{n+1},a_{n+2}]$.
For each $z\in S^1$, note that $\phi_z(x):=\phi(x,z)$ is a homeomorphism from the segment $[a_n,a_{n+1}]$ onto the segment $[a_{n+1},a_{n+2}]$.

We will first define $\phi_1$ in its whole domain, and then, by induction on $k$, the restriction of $\phi$ to $A_k$.

First define $\phi_1(x)$: for $x\in [a_n,a_{n+1}]$ let $\phi_1(x)=\lambda_{n+1}(\lambda^{-1}_{n}(x))$, that is, $\phi_1$ is affine, the image of $P$ under $f^k$ is equal to $(\lambda_k(1/2),1)$.

Also define $\phi_{z_{-1}}(\lambda_0(\nu_{-1}))=\lambda_1^{}(1/2)$, this is the only condition asked for this map. It is obvious that $\phi$ can be extended to the annulus $A_0$ so as to satisfy this unique condition (it is used, of course, that $z_{-1}\neq 1$).

Next let $k>0$, and define $\phi$ in $A_{k}$ assuming it is already known in $A_{k-1}$. Of course, $f$ is also defined in $A_0\cup\cdots\cup A_{k-1}$ and one can iterate $f^{k}$ at points in $A_0$, in particular the image of the point $(\lambda_0(\nu_{-k-1}),z_{-k-1})\in A_0$ under $f^{k}$ is a point having second coordinate $z_{-1}$, denote it by $(x_{k},z_{-1})$ in $A_k$.
Next extend $\phi$ to $A_{k+1}$ in order to satisfy only one condition: $\phi_{z_{-1}}(x_{k})=\lambda_{k+1}(1/2)$.

As was pointed out above, $f^n(P)=(\lambda_n(1/2),1)$ for every $n\geq 0$. It follows that
\begin{eqnarray*}
f^n(\lambda_0(\nu_{-n}),z_{-n})&=&f(f^{n-1}(\lambda_0(\nu_{-n}),z_{-n}))=f(x_{n-1},z_{-1})\\
& = & (\phi_{z_{-1}}(x_{n-1}),m_d(z_{-1}))=f^n(P),
\end{eqnarray*}
as required.

\qed

We finish this work with another negative result, negative in the direction of a possible classification of covering maps of the annulus.

Consider the Whitney (or strong) $C^0$ topology in the space of covering maps of the annulus, defined as follows: if $f\in Cov(A)$ and
$\epsilon: A\to \R^+$ is a continuous function, then the $\epsilon-$neighborhood of $f$ is
$$
\mathcal N_\epsilon(f)=\{g\in Cov(A)\ :\ d(g(x),f(x))<\epsilon(x)\ \forall x\in A\}
$$
where $d$ is any fixed distance compatible with the topology of $A$.
The notation for the space of $C^0$ maps endowed with this topology is $C^0_W(A)$.

\begin{defi}
\label{stable}
A map $f\in Cov(A)$ is $C_W^0(A)$-stable if there exists a $C^0_W$-neighborhood of $f$ such that every map $g$ in this neighborhood is conjugate to $f$.
\end{defi}

\begin{thm}
The map $p_d(z)=z^d$ is not $C^0_W(A)$-stable if $A$ is the punctured unit disc $A=D^*=\{z\in \C\ : \ 0<|z|<1\}$.
\end{thm}

Some remarks before the proof:\\
1. The distance $d$ is the Euclidean distance in $\D$.\\
2. If $g\in Cov(A)$ belongs to a neighborhood of $p_d$ such that $\epsilon(x)\to 0$ as $x\to\partial A$, then $g(z_n)\to p_d(z)$ whenever
$\{z_n\}$ is a sequence in $A$ converging to a point $z\in \partial A$. Thus $g$ extends continuously to the boundary, where it coincides
with $p_d$.\\
3. Note that $p_d$ is not $C^0$ stable, when one considers weak topology (define neigborhoods as above but with $\epsilon$ equal to a constant). This is obvious since one can create periodic points in $A$.\\
4. It is well known that the restriction of $p_d$ to its Julia set is $C^1$ stable. Moreover, the restriction of $p_d$ to $A=\C\setminus\{0\}$ is $C_W^1$ stable (see \cite{ipr}, Corollary 4).\\
5. Note that a homeomorphism having a hyperbolic attractor $\Lambda$ is $C_W^0$ stable when restricted to an invariant component of
$B(\Lambda)\setminus\Lambda$. Indeed, the definition of the conjugacy can be made in a fundamental domain and then extended to the future and the past. The same construction is not possible for noninvertible maps: note, for example, that the image of a fundamental domain is not necessarily a fundamental domain.\\
6. Note that the theorem also implies that if $A=\C\setminus\bar\D$, then $p_d$ is not $C_W^0(A)$ stable.\\

\begin{proof}
We will use $f=p_2$ and prove this case, the generalization to arbitrary degree being obvious.
Then the function to be perturbed is $f(x\exp(it))=x^2\exp(2it)$, where $x$ is positive and $t\in\R$. We will find a perturbation $g$ of $f$ having an invariant set with nonempty interior where $g$ is injective. This does not exist for $f$.

First the perturbation of $t\in\R\to 2t\in\R$.
Let $\rho$ and $\rho'$ be positive numbers such that $\rho'\leq \rho$. Then there exists an increasing continuous function $\phi=\phi_{\rho,\rho'}:[-\pi, \pi]\to\R$ satisfying $\phi(t+2\pi)= \phi (t) + 4 \pi$ for all $t\in \R$ such that $\phi((-\rho,\rho))= (-\rho',\rho')$, $\phi_0(0)=0$ and $|\phi(t)-2t|\leq 2\rho-\rho'$ for every $t$. Moreover, $\phi(t)=2t$ whenever $|t|>2\rho$. Moreover, one can ask the function $(t,\rho,\rho')\to\phi_{\rho,\rho'}(t)$ to be continuous.

Let $\epsilon$ be any positive continuous function defined in $\D^*$. Note that $\epsilon'\leq\epsilon$ implies that the $\epsilon'$-neighborhood of $f$ is contained in the $\epsilon$-neighborhood of $f$. So it can be assumed that the function $\epsilon$ satisfies $\epsilon(x\exp(it))=\epsilon(x)$.

It is claimed now that there exists a continuous function $\rho:(0,1)\to \R^+$ such that $2\rho(x)< \epsilon(x)$ for every $x\in(0,1)$ and $\rho(x^2)<\rho(x)$ for every $x\leq 1/2$. Indeed, first define $\rho(x)$ in the interval $[1/4,1/2]$ so that $\rho(x)<\epsilon(x)/2$ and $\rho(1/4)<\rho(1/2)$. Then define
$\rho$ for $x\in [1/16,1/4]$, so as to satisfy $\rho(x)<\rho(\sqrt{x})$ and $\rho(x)<\epsilon(x)/2$. Then use induction to define it in the remaining fundamental domains of the action of $x\to x^2$ in $(0,1/2]$. It is clear that $\rho$ can be continuously extended to the whole interval $(0,1)$ so as to satisfy $2\rho(x)<\epsilon(x)$.

Next proceed to the definition of $g$, a particular perturbation of $f$. Define
$$
g(x\exp(it))=x^2\exp(i\phi_{\rho(x),\rho(x^2)}(t)).
$$

Note first that $g$ is continuous, and defines a covering of the annulus, because the functions $\phi$ used in its definition are all increasing. Moreover, $z=x\exp(it)$ implies $f(z)=g(z)$ if $|t|>2\rho(x)$. Moreover:

\begin{eqnarray*}
|g(z)-f(z)| & = & x^2|\exp(i\phi_{\rho(x),\rho(x^2)}(t)-\exp(2it)| \leq x^2 |\phi_{\rho(x),\rho(x^2)}(t)-2t|\\
            & < & 2\rho(x)-\rho(x^2)<2\rho(x)<\epsilon(x)=\epsilon(z)
\end{eqnarray*}
Therefore $g$ belongs to the $\epsilon$-neighborhood of $f$. It remains to show that $g$ is not conjugate to $f$.
Note that the set $R:=\{x\exp(it)\ :\ x<1/2,\ |t|<\rho(x)\}$ is forward invariant under $g$, because $g(x\exp(it))=x^2\exp(i\phi(t))$
and $|\phi(t)|<\rho(x^2)$ if $|t|<\rho(x)$. But $g$ is injective on $R$, and $R$ has nonempty interior. Thus $f$ and $g$ cannot be conjugate.
\end{proof}

\section{Some final comments and questions.}

The problem of classifying coverings is very complicated, we cannot even imagine a classification of a neighborhood of the ''simplest''
map $p_d$.
One simple question, whose answer we still don't know is if every Whitney $C^0$ perturbation of $p_d$ has a covering by fundamental domains.
In the perturbation made above, the invariant foliation by circles centered at the origin is preserved.

It may be easy to see that no covering of $A$ of degree $|d|>1$ is $C^0_W(A)$ stable, but we won't give a proof of this here.

The question of the existence of periodic points for covering maps of the annulus will be considered in a following article. For example it will be proved that a covering of degree greater than one having an invariant continuum must have fixed points.

It is natural to consider the rotation number for coverings of the annulus as defined in section 2.1.
That is, given a lift $F:(0,1)\times \R\to (0,1)\times \R$ of $f$, take a point $(x_0,y_0)\in (0,1)\times\R$ and define
$$
\rho(x_0,y_0)=\lim_{n\to+\infty} \frac{y_n}{d^n},
$$
whenever this limit exists, and where $(x_n,y_n)=F^n(x_0,y_0)$. What conclusions can be drawn if, for example, this limit exists for every point? Does it necessarily define a continuous function?

\end{document}